\documentclass{aptpub}
\usepackage{bbm}

\authornames{Alexander Iksanov and Zakhar Kabluchko} 
\shorttitle{CLT and LIL for Biggins martingale} 

\DeclareMathOperator{\R}{\mathbb{R}}

\DeclareMathOperator{\V}{\mathbb{V}}
\DeclareMathOperator{\Prob}{\mathbb{P}}
\DeclareMathOperator{\Var}{Var}
\DeclareMathOperator{\Cov}{Cov}

\DeclareMathOperator{\me}{\mathbb{E}}
\DeclareMathOperator{\F}{\mathcal{F}}

\DeclareMathOperator{\1}{\mathbbm{1}}
\DeclareMathOperator{\D}{\mathfrak{D}}

\newcommand{\toasw}{\overset{a.s.w.}{\underset{n\to\infty}\longrightarrow}}
\newcommand{\todistr}{\overset{d}{\underset{n\to\infty}\longrightarrow}}
\newcommand{\toweak}{\overset{w}{\underset{n\to\infty}\longrightarrow}}

\newcommand{\M}{\mathbb{M}}

\newcommand{\mm}{\mathcal{Z}}
\newcommand{\mn}{\mathbb{N}}

\newcommand{\mr}{\mathbb{R}}
\newcommand{\lin}{\underset{n\to\infty}{\lim}}

\newcommand{\eps}{\varepsilon}
\newcommand{\eee}{{\rm e}}
\numberwithin{equation}{section}  

\begin{document}

\title{A central limit theorem and a law of the iterated logarithm for the Biggins martingale of the supercritical branching random walk} 

\authorone[Taras Shevchenko National University of Kyiv]{Alexander Iksanov} 
\authortwo[Westf\"{a}lische Wilhelms-Universit\"{a}t M\"{u}nster]{Zakhar Kabluchko}
\addressone{Faculty of Cybernetics, Taras Shevchenko National University of Kyiv, 01601 Kyiv, Ukraine} 
\addresstwo{Institut f\"{u}r Mathematische Statistik, Westf\"{a}lische Wilhelms-Universit\"{a}t M\"{u}nster, 48149 M\"{u}nster, Germany}
\begin{abstract}
Let $(W_n(\theta))_{n\in\mn_0}$ be the Biggins martingale associated with a supercritical branching random walk and denote by $W_\infty(\theta)$ its limit. Assuming essentially that the martingale $(W_n(2\theta))_{n\in\mn_0}$ is uniformly integrable and that $\Var W_1(\theta)$ is finite, we prove a functional central limit theorem for the tail process $(W_\infty(\theta) - W_{n+r}(\theta))_{r\in\mn_0}$  and
a law of the iterated logarithm for  $W_\infty(\theta)-W_n(\theta)$, as $n\to\infty$.
\end{abstract}



\section{Introduction and main results} \label{sec:Intro_and_main_results} 


\subsection{Introduction}
For several models of spin glasses it is known that the log-partition function has asymptotically Gaussian fluctuations in the high temperature regime. This was shown for the Sherrington--Kirkpatrick model in~\cite{Aizenman+Lebowitz+Ruelle:1987}, for the Random Energy Model and the $p$-spin model in~\cite{Bovier+Kurkova+Loewe:2002}, and for the Generalized Random Energy Model in~\cite{Kabluchko+Klimovsky:2014}, to give just an incomplete list of examples.  We are interested in the Biggins martingale $W_n(\theta)$ associated with a supercritical branching random walk (BRW), to be defined below.
With regard to the strength of its correlations, the branching random walk is located between the Random Energy Model and the Sherrington--Kirkpatrick model. Also, it can be thought of as a limiting case of the Generalized Random Energy Model. Since in all the three aforementioned models the log-partition function exhibits asymptotically Gaussian fluctuations at high temperatures, it is natural to expect that the branching random walk behaves similarly.
However, in the high-temperature regime (meaning that $\theta$ is small), the Biggins martingale $W_n(\theta)$ is, under appropriate conditions, uniformly integrable and converges almost surely (a.s.)\ to a limit $W_\infty(\theta)$ which is non-Gaussian. It follows that we cannot obtain a  Gaussian limit distribution whatever deterministic affine normalization we apply to $W_n(\theta)$.

In the present paper we prove a functional central limit theorem
(functional CLT) for the Biggins martingale $W_n(\theta)$ and its
logarithm under a natural \emph{random} centering. We also derive
a law of the iterated logarithm which complements the central
limit theorem.

Let us recall the definition of the branching random walk.  At
time $n=0$ consider an individual, the ancestor, located at the
origin of the real line.  At time $n=1$ the ancestor produces
offspring (the first generation) according to a point process $\mm
= \sum_{i=1}^J \delta_{X_i}$ on $\R$. The number of offspring, $J
= \mm(\R)$, is a random variable which is explicitly allowed to be
infinite with positive probability. The first generation produces
the second generation whose displacements with respect to their
mothers are distributed according to independent copies of the
same point process $\mm$.  The second generation produces the
third one, and so on. All individuals act independently of each
other.

More formally, let $\V=\cup_{n\in\mn_0}\mn^n$ be the set of all
possible individuals. The ancestor is identified with the empty
word $\varnothing$ and its position is $S(\varnothing)=0$. On some
probability space $(\Omega, \F, \Prob)$ let $(\mm(u))_{u \in \V}$
be a family of independent identically distributed (i.i.d.)\
copies of the point process $\mm$. An individual $u = u_1\ldots
u_n$ of the $n$th generation whose position on the real line is
denoted by $S(u)$ produces at time $n+1$ a random number $J(u)$ of
offspring which are placed at random locations on $\R$ given by
the positions of the point process $\sum_{i=1}^{J(u)} \delta_{S(u) + X_i(u)}$
where $\mm(u) =
\sum_{i=1}^{J(u)} \delta_{X_i(u)}$ and $J(u)$ is the number of points in $\mm(u)$. The offspring of the individual $u$ are
enumerated by $ui = u_1 \ldots u_n i$, where $i=1,\ldots,J(u)$ (if $J(u)<\infty$) or $i=1,2,\ldots$ (if $J(u)=\infty$), and the positions of the offspring are
denoted by $S(ui)$.  Note that no assumptions are imposed on the dependence structure of the random variables $J(u), X_1(u),X_2(u),\ldots$ for fixed $u\in\V$. The point process of
the positions of the $n$th generation individuals will be denoted
by $\mm_n$ so that $\mm_0=\delta_0$ and
$$
\mm_{n+1} = \sum_{|u|=n} \sum_{i=1}^{J(u)} \delta_{S(u)+X_i(u)},
$$
where, by convention, $|u|=n$ means that the sum is taken over all individuals of the $n$th generation rather than over all $u\in\mn^n$.
The sequence of point processes $(\mm_n)_{n \in
\mn_0}$ is then called a \emph{branching random walk} (BRW).

Throughout the paper, we assume that the BRW is \emph{supercritical}, that is $\me J>1$. In this case, the event $\mathcal{S}$ that the population survives has positive probability: $\Prob[\mathcal{S}]>0$.
Note that, provided that $J<\infty$ a.s.,\ the sequence $(\mm_n(\mr))_{n\in\mn_0}$ of generation
sizes in the BRW forms a Galton--Watson process.

An important tool in the analysis of the BRW is the Laplace transform
of the intensity measure $\mu := \me \mm$ of the point process $\mm$,
\begin{equation*}    \label{eq:m}
m: \R \to [0,\infty],   \qquad
\theta \mapsto \int_{\R} \eee^{-\theta x} \, \mu({\rm d} x) = \me \Bigg[ \int_{\R} \eee^{-\theta x} \,\mm({\rm d} x)\Bigg].
\end{equation*}
We make the standing assumption that $m(\gamma)< \infty$ for at least one $\gamma\in\R$, that is
$$
\D(m) := \{\theta \in \R\colon  m(\theta) < \infty\} \neq \varnothing.
$$
For $\gamma\in \D(m)$ define
\begin{equation*}\label{eq:W_n}
W_n(\gamma) :=    \frac1 {(m(\gamma))^n} \int_{\R} \eee^{-\gamma x} \,\mm_n({\rm d} x)
=
\frac 1 {(m(\gamma))^n}\sum_{|u|=n} (Y_u)^\gamma,
\quad
n\in\mn_0,
\end{equation*}
where $Y_u:=\eee^{-S(u)}$, and we recall that $S(u)$ is the position of the individual $u\in \V$. Let $\F_n$ be the
$\sigma$-field generated by the first $n$ generations of the BRW,
i.e. $\F_n=\sigma\{\mm(u)\colon  |u|<n\}$, where $|u|<n$ means that $u \in \mn^k$ for some $k<n$.
It is well-known and easy to check that, for every $\gamma\in \D(m)$, the sequence $(W_n(\gamma))_{n \in\mn_0}$
forms a non-negative martingale with respect to the filtration $(\F_n)_{n\in \mn_0}$ and thus converges a.s.\ to a
random variable which is denoted by $W_\infty(\gamma)$ and satisfies $\me W_\infty(\gamma)\leq 1$. This martingale is called the {\it Biggins martingale} or the {\it intrinsic martingale in the BRW}.
Possibly after the transformation $X_{i} \mapsto \gamma X_{i} +
\log m(\gamma)$ it is no loss of generality to assume that $\gamma=1$
and that
\begin{equation*}    \label{eq:m(1)=1}
m(1)    = \me \Bigg[\int_{\R} \eee^{-x} \,\mm({\rm d} x)\Bigg]   = \me \Bigg[\sum_{i=1}^J \eee^{-X_i}\Bigg]  = 1.
\end{equation*}

\subsection{Central limit theorem}\label{cltres}
Let $\R^{\infty}$ be the space of infinite sequences $x = (x_0,x_1,x_2,\ldots)$ with $x_j\in\R$ for all $j\in\mn_0$. Endow $\R^{\infty}$ with a complete, separable metric
$$
\rho(x,y) = \sum_{j=0}^{\infty} 2^{-j} \frac{|x_j-y_j|}{1+|x_j-y_j|},
\quad
x,y\in \R^{\infty}
$$
which metrizes the pointwise convergence.
\begin{thm}\label{main2}
Suppose that $m(1)=1$, $\sigma^2 := \Var W_1(1) < \infty$ and $m(2)<1$.  Then,
\begin{equation}\label{clt}
\left( \frac{W_\infty(1) -
W_{n+r}(1)}{(m(2))^{(n+r)/2}}\right)_{r\in \mn_0} \toweak
\left(\sqrt{v^2 W_\infty(2)} \, U_r\right)_{r\in\mn_0}
\end{equation}
weakly on $\R^{\infty}$, where $v^2:= \Var W_\infty(1) = \sigma^2(1-m(2))^{-1}$, and  $(U_r)_{r\in\mn_0}$ is a stationary zero-mean Gaussian sequence which is independent of $W_\infty(2)$ and has the covariance function
$$
\Cov (U_r, U_s) =   (m(2))^{|r-s|/2}, \quad r,s\in\mn_0.
$$
\end{thm}
Note that $(U_r)_{r\in\mn_0}$ can be viewed as an $\text{AR}(1)$-process or as an Ornstein--Uhlenbeck process sampled at nonnegative integer times. In the case when the martingale $(W_n(2))_{n\in\mn_0}$ is not uniformly integrable (and hence, $W_\infty(2)=0$), Theorem~\ref{main2} is still valid, but the limiting process in~\eqref{clt} is trivial.
Specifying Theorem~\ref{main2} to $r=0$, we obtain the following central limit theorem for the tail of the Biggins martingale.

\begin{cor}\label{clt1}
Suppose that $m(1)=1$, $\Var W_1(1) < \infty$ and $m(2)<1$. Then,
$$
\frac{W_\infty(1) - W_{n}(1)}{(m(2))^{n/2}} \todistr
{\rm N} (0,v^2 W_\infty(2)),
$$
where the limiting distribution is a scale mixture of normals with randomized variance $v^2 W_\infty(2)$.
\end{cor}

In fact, we shall prove a result with a mode of convergence stronger than in Theorem~\ref{main2}. Let $\xi:\Omega\to E$ be a random variable on $(\Omega, \mathcal F, \Prob)$ with values in a Polish space $E$, and let $\mathcal G \subset \mathcal F$ be a $\sigma$-field. Denote by $\M(E)$ the space of probability measures on $E$ endowed with the topology of weak convergence. A random variable of the form $L:\Omega \to \M(E)$ is called a \emph{Markov kernel} or a \emph{probability transition kernel}.
The conditional law of $\xi$ given $\mathcal G$ is defined as a $\mathcal G$-measurable mapping $L: \Omega \to \M(E)$ such that for every random event $A\in \mathcal G$ and every bounded Borel function $f: E\to\R$, we have
$$
\me [f(\xi) \1_{A}] = \int_A \left(\int_E f(x) L (\omega; {\rm d} x) \right)\Prob({\rm d} \omega).
$$
It is known that $L$ is defined uniquely up to sets of probability $0$. A sequence of Markov kernels $L_n:\Omega \to \M(E)$ converges to a Markov  kernel $L_\infty:\Omega \to \M(E)$ in the \emph{almost surely weak} (a.s.w.)\ sense if the set of $\omega\in\Omega$ for which the probability measure $L_n(\omega)$ converges to $L_\infty(\omega)$ weakly on $E$ has probability $1$. We refer to~\cite{Gruebel+Kabluchko:2014} for the basic properties of the a.s.w.\ convergence and its relations to other modes of convergence (including the weak and the stable convergence).

\begin{thm}\label{main2_asw}
Suppose that $m(1)=1$, $\Var W_1(1) < \infty$ and $m(2)<1$.
Denote by $L_n:\Omega \to \M(\R^{\infty})$ the conditional law of
the process
$$
\left( \frac{W_\infty(1) - W_{n+r}(1)}{(m(2))^{(n+r)/2}} \right)_{r\in \mn_0}
$$
given the $\sigma$-field $\F_n$ and viewed as a random variable on $(\Omega, \F,\Prob)$ with values in $\M(\R^{\infty})$. Then, $L_n$ converges almost surely weakly to the Markov kernel
$$
L_\infty:\Omega \to \M(\R^{\infty}), \quad \omega \mapsto \mathcal L \left\{ \left(\sqrt{v^2 W_\infty(2;\omega)} \, U_r\right)_{r\in\mn_0 }\right\},
$$
where $\mathcal L\{\cdot\}$ denotes the probability law of a
process, and $(U_r)_{r\in\mn_0}$ is a discrete-time
Ornstein--Uhlenbeck process as in Theorem \ref{main2} but defined
on some probability space other than $(\Omega, \F, \Prob)$.
\end{thm}

It follows from Proposition~4.6 and Remark~4.7 of~\cite{Gruebel+Kabluchko:2014} that the weak convergence in Theorem~\ref{main2} is a consequence of the a.s.w.\ convergence in Theorem~\ref{main2_asw}. Hence, we only need to prove Theorem~\ref{main2_asw}. This will be done in Section~\ref{sec:proof_clt}.
Specifying Theorem~\ref{main2_asw} to $r=0$ we obtain the following almost surely weak version of Corollary~\ref{clt1}.
\begin{cor}
Suppose that $m(1)=1$, $\Var W_1(1) < \infty$ and $m(2)<1$. Then, we have the following almost surely weak convergence of Markov kernels from $\Omega$ to $\M(\R)$:
\begin{equation}\label{eq:cor_clt_asw}
\mathcal L \left\{\frac{W_\infty(1) - W_n(1)}{(m(2))^{n/2}}\Big |\F_n \right\} \toasw
\left\{\omega \mapsto {\rm N}\left(0, v^2 W_\infty(2;\omega)\right) \right\}.
\end{equation}
\end{cor}

We can also derive a central limit theorem for the ``log-partition function'' $\log W_n(1)$.
\begin{cor}
Suppose that $m(1)=1$, $\sigma^2 = \Var W_1(1) < \infty$,
$m(2)<1$, and that the survival event $\mathcal S$ has probability
$1$. Then, we have the following almost surely weak convergence of
Markov kernels from $\Omega$ to $\M(\R)$:
\begin{equation}\label{eq:cor_clt_log_asw}
\mathcal L \left\{\frac{\log W_\infty(1) - \log
W_n(1)}{(m(2))^{n/2}}\Big |\F_n
\right\}
\toasw
\left\{\omega \mapsto {\rm N}\left(0,
v^2\frac{W_\infty(2;\omega)}{W_\infty^2(1;\omega)}\right)
\right\}.
\end{equation}
\end{cor}
\begin{proof}
Dividing the Markov kernels on both sides
of~\eqref{eq:cor_clt_asw} by $W_n(1)$ (which is $\F_n$-measurable)
and using the fact that $\lim_{n\to\infty} W_n(1) = W_\infty(1)
>0$ a.s.\ on $\mathcal{S}$ (for the positivity, see
the implication (ii) $\Rightarrow$ (i) on p.~218 in
\cite{Lyons:1997}) together with the Slutsky lemma, we obtain that
\begin{equation}\label{eq:clt_log_wspom}
\mathcal L \left\{\frac 1{(m(2))^{n/2}}
\left(\frac{W_\infty(1)}{W_n(1)} - 1\right)\Big |\F_n
\right\} \toasw
\left\{\omega \mapsto {\rm N}\left(0, v^2
\frac{W_\infty(2;\omega)}{W_\infty^2(1;\omega)}\right)
\right\}.
\end{equation}
It is easy to check that if $(\xi_n)_{n\in\mn_0}$ is a sequence of random variables such that $a_n^{-1}\xi_n$ converges in distribution to some $\xi$ as $n\to\infty$, where $a_n\to 0$ is a deterministic sequence, then $a_n^{-1} \log (1+\xi_n)$ converges in distribution to the same limit $\xi$. Applying this to~\eqref{eq:clt_log_wspom} pointwise yields~\eqref{eq:cor_clt_log_asw}.
\end{proof}

A central limit theorem for the tail martingale of a
Galton--Watson process was obtained by Athreya~\cite{Athreya:1968}
(who considered multitype branching processes) and
Heyde~\cite{Heyde:1970} (who also treated the case when the limit
is $\alpha$-stable in~\cite{Heyde:1971}). This CLT can also be
found on page 55 of the book~\cite{Athreya:2004}. In the more
general setting of multitype branching processes, related CLT's
were obtained in~\cite{Asmussen+Keiding:1978, Athreya:1968,
Kesten+Stigum:1966}. A functional CLT for the tail martingale was
obtained by Heyde and Brown~\cite{Heyde+Brown:1971}. By
considering BRW with trivial displacements (see the proof of
Corollary~\ref{cor2} for more details), the results of
Section~\ref{cltres} can be used to recover most of the results
obtained in~\cite{buhler:1969,Heyde:1970,Heyde+Brown:1971}. Linear
statistics of branching diffusion processes and superprocesses are
objects of current active studies; see, e.g.,\
\cite{Adamczak+Milos:2015} and~\cite{Ren+Song+Zhang:2015}.
Although the Biggins martingale is a special case of linear
statistics, the conditions imposed
in~\cite{Adamczak+Milos:2015,Ren+Song+Zhang:2015} exclude test
functions of the form $x\mapsto \eee^{-x}$. In the setting of
weighted branching processes (which includes BRW as a special
case), a CLT was obtained in~\cite{Roesler+Topchii+Vatutin:2002},
however the moment conditions
of~\cite{Roesler+Topchii+Vatutin:2002} are slightly more
restrictive than ours. Also, we provide a functional CLT and a
stronger (a.s.w.)\ mode of convergence. Recently, CLT's for tail
martingales associated with random trees (and related to the
derivative of the Biggins martingale at $0$) were proved
in~\cite{Gruebel+Kabluchko:2014, Neininger:2013, Sulzbach:2015}.

\subsection{Law of the iterated logarithm}\label{lilres}

The law of the iterated logarithm given next complements the central limit theorem given in Corollary \ref{clt1}.
\begin{thm}\label{main1}
Assume that $m(1)=1$, $\sigma^2 = \Var W_1(1)<\infty$, $\me W_1(2)\log^+W_1(2)<\infty$, and that the function $r\to (m(r))^{1/r}$ is finite and decreasing on $[1,2]$ with
\begin{equation}\label{mean}
{-\log m(2)\over 2}<-{m^\prime(2)\over m(2)},
\end{equation}
where $m^\prime$ denotes the left derivative. Then, $W_\infty(1)$ and $W_\infty(2)$ are positive almost surely on the survival set $\mathcal{S}$,  and
\begin{align}
&\underset{n\to\infty}{\lim\sup} \frac {W_\infty(1)-W_n(1)} {\sqrt{(m(2))^n\log n}} =\sqrt{2v^2W_\infty(2)},\label{lil} \\ 
&\underset{n\to\infty}{\lim\inf} \frac {W_\infty(1)-W_n(1)}{\sqrt{(m(2))^n\log n}}=-\sqrt{2v^2W_\infty(2)}\label{lil1}
\end{align}
almost surely, where $v^2 = \Var W_\infty(1) =
\sigma^2(1-m(2))^{-1}<\infty$.
\end{thm}
\begin{rem}\label{imp}
It is well-known (see Theorem A in \cite{Biggins:1977}, p.~218 in
\cite{Lyons:1997} or Theorem 1.3 in \cite{Alsmeyer+Iksanov:2009})
that conditions $\me W_1(2)\log^+ W_1(2)<\infty$ and \eqref{mean}
ensure the uniform integrability of $(W_n(2))_{n\in\mn_0}$ which
particularly implies that $W_\infty(2)$ is a.s.\ positive on
$\mathcal{S}$.
\end{rem}
\begin{rem}
Actually, under the assumptions $m(1)=1$ and $m(2)<+\infty$, the
conditions $\me W_1(2)\log^+W_1(2)<\infty$ and~\eqref{mean} are
also necessary for the uniform integrability of
$(W_n(2))_{n\in\mn_0}$. Indeed, the function $m(\theta)$ is convex
on the interval $[1,2]$, hence it has left derivative $m'(2)\in
(-\infty, +\infty]$. With this at hand the uniform integrability
implies~\eqref{mean} by Theorem~1.3
in~\cite{Alsmeyer+Iksanov:2009}. It is not possible that
$m'(2)=+\infty$ because, together with $m(2)<\infty$, this would
contradict~\eqref{mean}. Hence, $m'(2)$ is finite. Under this
condition,  the uniform integrability of $(W_n(2))_{n\in\mn_0}$
implies that $\me W_1(2)\log^+W_1(2)<\infty$ by Theorem~1.3
in~\cite{Alsmeyer+Iksanov:2009}.
\end{rem}
\begin{rem}
It will be shown in~\eqref{eq:cov_tail} that $$\Var [W_\infty(1) - W_n(1)] = v^2 (m(2))^n.$$
In~\eqref{lil} and~\eqref{lil1} it is possible to replace $\log n$ by the asymptotically equivalent expression $\log \log (v^2(m(2))^n)$, thereby justifying the use of the term ``law of the iterated logarithm''. Therefore, the normalization in~\eqref{lil} and~\eqref{lil1} is very similar to that in the classical law of the iterated logarithm, but it should be stressed that unlike in the classical case, the limits in~\eqref{lil} and~\eqref{lil1} are random.
\end{rem}

As an immediate consequence of Theorem \ref{main1} we derive a previously known result (see \cite{Heyde+Leslie:1971} and Theorem 3.1 (ii) on p.~28 in \cite{Asmussen+Hering:1983}) concerning the Galton--Watson process.
\begin{cor}\label{cor2}
Consider a Galton--Watson process $(Y_n)_{n\in\mn_0}$ with $m := \me
Y_1 \in (1,\infty)$ and $s^2:=\Var Y_1 <\infty$. Then, for the
martingale $W_n := Y_n/m^n$ and its almost sure limit $W_\infty$
we have
\begin{align*}
\underset{n\to\infty}{\lim\sup}{m^{n/2}(W_\infty - W_n) \over \sqrt{\log n}} = \sqrt{2v^2W_\infty},\\
\underset{n\to\infty}{\lim\inf}{m^{n/2}(W_\infty-W_n)\over
\sqrt{\log n}}=-\sqrt{2v^2W_\infty}
\end{align*}
almost surely, where $v^2:=\Var W_\infty = s^2(m(m-1))^{-1}$.
\end{cor}
\begin{proof}
Consider a BRW in which the genealogical structure is the same as in $(Y_n)_{n\in\mn_0}$, and the displacements of all individuals are deterministic and equal to $\log m$. That is, $\eee^{-X_i}=m^{-1}$ for $i=1,\ldots, Y_1$ and we have, for
$\gamma>0$,
$$
m(\gamma) = m^{1-\gamma}
\quad\text{and}\quad
W_n(\gamma) = Y_n/m^n = W_n,
\quad   n\in\mn_0.
$$
Hence $m(1)=1$, $W_\infty=W_\infty(2)$, $\Var W_1=m^{-2}s^2$ and
$\Var W_\infty = (m(m-1))^{-1}s^2$. The assumptions of
Theorem~\ref{main1} are easy to verify, whence the result.
\end{proof}

Plainly, Theorem \ref{main1} is a result on the rate of the a.s.\ convergence of $W_n(1)$ to its limit. There have already been several works that investigated how fast $W_n(1)$ approaches $W_\infty(1)$ in various senses, see \cite{Iksanov:2006, Iksanov+Meiners:2010} for the rate of a.s.\ convergence, \cite{Alsmeyer+Iksanov+Polotskiy+Roesler:2009} for the rate of $L_p$-convergence. Laws of the iterated logarithm for martingales related to path length of random trees were obtained in~\cite{Sulzbach:2015}. We also refer to~\cite{Hall+Heyde:1980} for general central limit theorems and laws of the iterated logarithm for martingales not necessarily related to branching processes.

\section{Proof of Theorem~\ref{main2_asw}}\label{sec:proof_clt}

Throughout the rest of the paper, we shall use $W_n$ and
$W_\infty$ as shorthands for $W_n(1)$ and $W_\infty(1)$. Note that
$W_n(2)$ and $W_\infty(2)$ retain their meaning.

For any $u \in \V$ let $W_r^{(u)}$ and $W_\infty^{(u)}$,  $r\in\mn_{0}$, be the analogues of $W_r$ and $W_\infty$, $r\in\mn_{0}$, but based on the progeny of the individual $u$ rather than the progeny of the initial ancestor $\varnothing$. That is,
$$
W_{r}^{(u)} = \sum_{|v|=r} \eee^{-(S(uv) - S(u))}, \quad r\in \mn_0,
\quad
\text{ and }
\quad
W_{\infty}^{(u)} = \lim_{r\to\infty} W_r^{(u)} \quad \text{a.s.}
$$
Recall the notation $Y_u= \eee^{-S(u)}$. We shall frequently use the decomposition
$$
W_{n+r} = \sum_{|u|=n} Y_u W_r^{(u)}, \quad r\in \mn_0\cup\{\infty\}.
$$
Observe that for $|u|=n$,  the $Y_u$'s are $\mathcal{F}_n$-measurable, whereas the $W_r^{(u)}$'s are independent of $\mathcal{F}_n$.
We need two results on the covariance structure of the martingale $(W_n)_{n\in\mn_0}$.
\begin{prop}\label{prop:var_cov}
Under the assumptions  $m(2)<1$ and $\sigma^2 = \Var W_1 <\infty$ we have
\begin{equation}\label{variance}
\Var  W_r=\sigma^2(1+m(2)+\ldots+(m(2))^{r-1}),\quad r\in\mn.
\end{equation}
Furthermore, the martingale $(W_n)_{n\in\mn_0}$ converges in $L_2$ (and a.s.)\ to $W_\infty$ which satisfies
$$
\Var  W_\infty = \frac{\sigma^2}{1-m(2)}.
$$
In particular, $(W_n)_{n\in\mn_0}$ is uniformly integrable and $W_\infty>0$ a.s.\ on $\mathcal{S}$.
\end{prop}
\begin{proof}
We shall check \eqref{variance} by using mathematical induction.
The formula holds for $r=1$ because $\Var W_1 = \sigma^2$. Suppose \eqref{variance} holds for some $r\in\mn$. Then
\begin{eqnarray*}
\Var W_{r+1}
&=&
\me\Bigg[\bigg(\sum_{|u|=r}Y_u W_1^{(u)}\bigg)^2\Bigg]-1\\
&=&
\me\Bigg[\sum_{|u|=r}Y_u^2 (W_1^{(u)})^2\Bigg] + \me \Bigg[\me \Bigg[\sum_{\substack{|u|=|v|=r\\u\neq v}}Y_uY_vW_1^{(u)}W_1^{(v)}\bigg|\mathcal{F}_r\Bigg]\Bigg]-1\\
&=&
(m(2))^r (\sigma^2+1) + \me \Bigg[\sum_{\substack{|u|=|v|=r\\u\neq v}}Y_uY_v\Bigg]-1\\
&=&
\sigma^2 (m(2))^r + \Var W_r
\end{eqnarray*}
because
$$
\me \Bigg[\sum_{\substack{|u|=|v|=r\\ u\neq v}} Y_uY_v \Bigg] - 1
=
\me\Bigg[\bigg(\sum_{|u|=r}Y_u\bigg)^2\Bigg]-1-
\me\Bigg[\sum_{|u|=r}Y_u^2\Bigg]
=\Var  W_r - (m(2))^r.
$$
This completes the induction and proves \eqref{variance}. Since $m(2)<1$, the martingale $(W_n)_{n\in\mn_0}$ is bounded in $L^2$ and hence, converges in $L^2$ to $W_\infty$. In particular,  $(W_n)_{n\in\mn_0}$ is uniformly integrable and $W_\infty>0$ a.s.\ on $\mathcal{S}$.   Letting $r$ in \eqref{variance} tend to $\infty$ we infer $\Var W_\infty = \sigma^2(1-m(2))^{-1}$.
\end{proof}

\begin{cor}
The random variables $W_{r+1}-W_r$, $r\in\mn_0$, are uncorrelated and
\begin{equation}\label{var_increments}
\Var[W_{r+1}-W_r] = \sigma^2 (m(2))^r.
\end{equation}
\end{cor}
\begin{proof}
The increments $W_{r+1}-W_r$, $r\in\mn_0$, are uncorrelated just
because $(W_n)_{n\in\mn_0}$ is a martingale. We thank the referee
for this observation that enabled us to simplify our original
argument. Further, we have, for $r<s$,
$$
\me[(W_s-W_r)^2] = \me\Bigg[\bigg(\sum_{|u|=r} Y_u (W_{s-r}^{(u)}-1)\bigg)^2\Bigg] = (m(2))^r \Var W_{s-r}
=
\sigma^2 \sum_{k=r}^{s-1}(m(2))^k.
$$
This proves~\eqref{var_increments} by taking $s=r+1$.
\end{proof}

\begin{proof}[Proof of Theorem~\ref{main2_asw}]
The conditional law $L_n$ can be explicitly described as follows. On some probability space $(\tilde \Omega,  \tilde{\mathcal  F}, \tilde \Prob)$ (which is different from the probability space $(\Omega, \F, \Prob)$ on which the BRW is defined) we construct a family $(\tilde W_n^{(u)})_{n\in \mn_0\cup\{\infty\}}$, $u\in \V$, of independent (for different $u$'s) distributional copies of the stochastic process $(W_n(1))_{n\in\mn_0\cup\{\infty\}}$.
For every $\omega\in\Omega$ let $U_{n,r}(\omega)$ be  random variables on the space $(\tilde \Omega,  \tilde{\mathcal  F}, \tilde \Prob)$ defined by
\begin{equation}\label{u_n_r}
U_{n,r}(\omega) := \frac{\sum_{|u|=n} Y_u(\omega)(\tilde W_\infty^{(u)} - \tilde W_r^{(u)})}{(m(2))^{(n+r)/2}},
\quad n,r\in\mn_0.
\end{equation}
With this notation, the conditional law $L_n: \Omega \to \M(\R^{\infty})$ is the Markov kernel
$$
L_n(\omega) = \mathcal L \left\{ (U_{n,r}(\omega))_{r\in\mn_0}\right\},
$$
where $\mathcal L$ is the law taken with respect to the probability distribution $\tilde \Prob$. Recall also that the Markov kernel $L_\infty: \Omega \to \M(\R^{\infty})$ is defined by
$$
L_\infty(\omega) = \mathcal L \left\{ \left(\sqrt{v^2 W_\infty(2;\omega)} \, U_r\right)_{r\in\mn_0}\right\}.
$$

Weak convergence of probability measures on $\R^{\infty}$ is equivalent to the weak convergence of their finite-dimensional distributions. So, we need to prove that for $\Prob$--a.e.\ $\omega\in \Omega$ we have $L_n(\omega) \to L_\infty(\omega)$ in the sense of finite-dimensional distributions. We take any $r_1,\ldots,r_d\in \mn_0$ and show that for $\Prob$--a.e.\ $\omega\in \Omega$,
\begin{equation}\label{conv_fidi}
(U_{n,r_1}(\omega), \ldots, U_{n, r_d}(\omega)) \todistr \sqrt{v^2 W_\infty(2;\omega)}  (U_{r_1},\ldots, U_{r_d}).
\end{equation}
This is done by verifying the conditions of the $d$-dimensional Lindeberg central limit theorem. Clearly, \eqref{u_n_r} provides a representation of the vector $(U_{n,r_1}(\omega), \ldots, U_{n, r_d}(\omega))$ as a sum of independent but not identically distributed random vectors. (Note that $Y_u(\omega)$ are treated as constants).
For every $r,n\in\mn_0$ and $\omega\in \Omega$ we have  $\me[U_{n,r}(\omega)] = 0$ and
\begin{align*}
\lefteqn{\me [U_{n,r}(\omega)U_{n,s}(\omega)]}\\
&=
\frac 1 {(m(2))^{n + (r+s)/2}} \sum_{|u|=|v|=n} Y_u(\omega) Y_v(\omega) {\rm Cov} (\tilde W_\infty^{(u)}-\tilde W_r^{(u)}, \tilde W_\infty^{(v)} - \tilde W_s^{(v)})  \\
&=
\frac 1 {(m(2))^{n + (r+s)/2}}   \sum_{|u|=n} Y_u^2(\omega) {\rm Cov} (W_\infty - W_r, W_\infty - W_s)  \\
&= \frac {v^2} {(m(2))^{n}} \left(\sum_{|u|=n}
Y_u^2(\omega)\right)(m(2))^{|r-s|/2},
\end{align*}
where we used that $\tilde W_r^{(u)}$ and $\tilde W_s^{(v)}$ are
independent for $u\neq v$ and the formula
\begin{align}\label{eq:cov_tail}
\Cov (W_\infty - W_r, W_\infty - W_s)
=
\frac{\sigma^2}{1-m(2)} (m(2))^{\max\{r,s\}}
\end{align}
which follows from Corollary~\ref{var_increments}. By letting $n\to\infty$ it follows that for $\Prob$-a.e.\ $\omega\in\Omega$,
\begin{align*}
\lim_{n\to\infty} \me \left[U_{n,r}(\omega)U_{n,s}(\omega)\right] =
v^2 W_\infty(2;\omega)(m(2))^{|r-s|/2} = v^2 W_\infty(2;\omega)
\Cov(U_r, U_s).
\end{align*}
This verifies the convergence of covariances in~\eqref{conv_fidi}.
It remains to check the Lindeberg condition for $\Prob$-a.e.\ $\omega\in\Omega$. This can be done individually for each component of the vectors in~\eqref{conv_fidi}. For every $\eps>0$, we have
\begin{align*}
L_n(\eps)
&:=
\sum_{|u|=n} \me \left[\frac{Y_u^2(\omega) (\tilde W_\infty^{(u)} - \tilde W_\infty^{(u)})^2}{(m(2))^{n+r}}
\1_{\left\{\frac{Y_u^2(\omega) (\tilde W_\infty^{(u)} - \tilde W_\infty^{(u)})^2}{(m(2))^{n+r}}>\eps^2\right\}}\right]\\
&=
\frac 1 {(m(2))^{n+r}} \sum_{|u|=n} Y_u^2(\omega) \me \left[(W_\infty-W_r)^2 \1_{\left\{\frac{(W_\infty-W_r)^2}{(m(2))^r} > \frac{\eps^2}{Y_u^2(\omega)/(m(2))^n} \right\}} \right]\\
&\leq
\frac 1 {(m(2))^{n+r}} \left(\sum_{|u|=n} Y_u^2(\omega)\right) G_r\left(\frac{\eps^2}{\sup_{|u|=n}Y_u^2(\omega)/(m(2))^n}\right),
\end{align*}
where
$$
G_r(A) = \me \left[(W_\infty-W_r)^2 \1_{\left\{\frac{(W_\infty-W_r)^2}{(m(2))^r} > A\right\}}\right], \quad A>0.
$$
Since the second moment of $W_\infty-W_r$ is finite, we have $\lim_{A\to +\infty} G_r(A) = 0$. By Theorem 3 in \cite{Biggins:1998}, the assumption $m(2)<\infty$  ensures that
$$
\lin  \frac 1 {(m(2))^n}\underset{|u|=n}{\sup} Y_u^2(\omega) = 0 \quad \text{for $\Prob$-a.e. } \omega\in\Omega.
$$
Also, for $\Prob$-a.e.\ $\omega\in\Omega$,
$$
\lim_{n\to\infty} \frac 1 {(m(2))^{n+r}} \sum_{|u|=n} Y_u^2(\omega) = \frac 1 {(m(2))^{r}} W_\infty(2;\omega).
$$
It follows that
$$
\lin L_n(\eps) =  0 \quad \text{for $\Prob$-a.e. } \omega\in\Omega.
$$
An application of the multidimensional Lindeberg CLT completes the proof of~\eqref{conv_fidi}.
\end{proof}

\section{Proof of Theorem \ref{main1}}

Since relations \eqref{lil} and \eqref{lil1} trivially hold on
$\mathcal{S}^c$, we have to prove that these hold a.s.\ on
$\mathcal{S}$.

We start by recalling that, according to Remark \ref{imp},
$W_\infty(2)>0$ a.s.\ on $\mathcal{S}$. The proof follows the
pattern of the proof of Theorem 3.1 on p.~28 in
\cite{Asmussen+Hering:1983}. Recall the notation $W_n=W_n(1)$ and
$W_\infty=W_\infty(1)$. We only treat the upper limit.
Investigating $W_n-W_\infty$ rather than $W_\infty-W_n$
immediately gives the result for the lower limit. Also, without
loss of generality we assume in what follows that
$\Prob[\mathcal{S}]=1$ (otherwise we have to use Lemma \ref{asmus}
below with the probability measure $\Prob$ replaced with
$\Prob(\cdot|\mathcal{S})$ and write ``a.s.\ on the survival set
$\mathcal S$'' rather than ``a.s.'' throughout). This assumption
ensures that $W_\infty$ and $W_\infty(2)$ are positive a.s.\
rather than with positive probability.

We shall use the following representations
\begin{equation}\label{repr1}
W_\infty - W_n=\sum_{|u|=n}Y_u (W_\infty^{(u)}-1)
\quad\text{and}\quad
W_{n+r} - W_n=\sum_{|u|=n} Y_u (W_r^{(u)}-1)
\end{equation}
for $r\in\mn$. By the reasons that will become clear in a while we
first consider the sums as above with truncated summands. It will
be convenient to write $\eee^a$ for $m(2)^{-1/2}$. For  $u\in\V$
with $|u|=n \in \mn_0$ and $r\in\mn_\infty:=\mn\cup\{\infty\}$,
put
$$
Z^{(u)}_{n,r} := Y_u(W_r^{(u)}-1)\1_{\{\eee^{an} Y_u|W_r^{(u)}-1|\leq 1\}}
$$
and then
\begin{equation}\label{repr2}
V_{n,r}=\sum_{|u|=n}(Z^{(u)}_{n,r}-\me[Z^{(u)}_{n,r}|\mathcal{F}_n]).
\end{equation}
\begin{lem}
For $r\in\mn_\infty$,
\begin{equation}\label{important}
\lin \eee^{2an}\Var [V_{n,r}|\mathcal{F}_n] = W_\infty(2)\Var W_r \quad\text{a.s.}
\end{equation}
\end{lem}
\begin{proof}
Conditionally on $\F_n$, the random variables $Z_{n,r}^{(u)}$,
$|u|=n$, are independent (but not identically distributed). By
definition of $V_{n,r}$ we have
$$
\Var [V_{n,r}|\mathcal{F}_n]
=
\sum_{|u|=n} \me \big[(Z_{n,r}^{(u)})^2|\F_n\big] - \sum_{|u|=n} \big(\me \big[Z_{n,r}^{(u)}|\F_n\big]\big)^2
=: T_{n,r}' -T_{n,r}''.
$$
To verify~\eqref{important}, we are going to show that
\begin{align}
&\lin \eee^{2an} T_{n,r}' = W_\infty(2)\Var W_r \quad\text{a.s.},\label{tech1}\\
&\lin \eee^{2an} T_{n,r}''=0 \quad\text{a.s.} \label{tech2}
\end{align}
\noindent
\emph{Proof of~\eqref{tech1}.}
Let $F_r(x):=\Prob\{|W_r-1|\leq x\}$, $x\geq 0$, be the distribution function of $|W_r-1|$. With this notation, we have
$$
T_{n,r}' := \sum_{|u|=n} \me \big[(Z_{n,r}^{(u)})^2|\F_n\big] =
\sum_{|u|=n} \Bigg(Y_u^2 \int_{[0,\,\eee^{-an}Y_u^{-1}]}x^2{\rm
d}F_r(x)\Bigg)
$$
and thereupon
\begin{equation}\label{eq:wspom}
\left(\sum_{|u|=n}Y_u^2\right) \int_{[0,\,
(\eee^{an}\underset{|u|=n}{\sup} Y_u)^{-1}]}x^2{\rm d}F_r(x) \leq
T_{n,r}' \leq \left(\sum_{|u|=n} Y_u^2\right) \Var W_r.
\end{equation}
By Theorem 3 in \cite{Biggins:1998}, the assumption $m(2)<\infty$ alone ensures that
\begin{equation}\label{sup_Y_u_est}
\lin \eee^{an} \underset{|u|=n}{\sup}Y_u=0 \quad \text{a.s.}
\end{equation}
Thus, the integral in the lower estimate in~\eqref{eq:wspom}
converges a.s.\ to $\Var W_r$. To complete the proof
of~\eqref{tech1} we recall that
\begin{equation}\label{eq:Y_u_square}
\lin \eee^{2an} \sum_{|u|=n} Y_u^2 = W_\infty(2) \quad \text{a.s}.
\end{equation}

\noindent \emph{Proof of~\eqref{tech2}.} Since $\me
[W_r^{(u)}-1]=0$,
\begin{align*}
T_{n,r}''
&= \sum_{|u|=n} Y_u^2 \big(\me \big[(W_r^{(u)}-1) \1_{\{\eee^{an} Y_u|W_r^{(u)}-1|\leq 1\}} \big]\big)^2\\
&= \sum_{|u|=n} Y_u^2 \big(\me \big[(W_r^{(u)}-1) \1_{\{\eee^{an} Y_u|W_r^{(u)}-1| >  1\}} \big]\big)^2.
\end{align*}
Using $W_r^{(u)}-1\leq |W_r^{(u)}-1|$ gives
\begin{align*}
T_{n,r}''
&\leq
\sum_{|u|=n} \Bigg(Y_u^2 \Bigg(\int_{\eee^{-an}Y_u^{-1}}^\infty x {\rm d} F_r(x)\Bigg)^2\Bigg)\\
&\leq
\Bigg(\sum_{|u|=n} Y_u^2\Bigg) \Bigg(\int_{(\eee^{an} \underset{|u|=n} \sup Y_u)^{-1}}^{\infty} x {\rm d} F_r(x)\Bigg)^2.
\end{align*}
Since $\int_{0}^{\infty} x {\rm d} F_r(x)$ is finite, the integral
on the right-hand side converges a.s.\ to $0$ as $n\to\infty$
by~\eqref{sup_Y_u_est}. Recalling~\eqref{eq:Y_u_square}, we arrive
at ~\eqref{tech2}. Taken together, \eqref{tech1} and~\eqref{tech2}
yield \eqref{important}.
\end{proof}

The main tool in the proof of Theorem~\ref{main1} is the following lemma, see Proposition 7.2 on p.~436 in \cite{Asmussen+Hering:1983}.
\begin{lem}\label{asmus}
Let $(\mathcal{G}_n)_{n\in\mn_0}$ be an increasing sequence of $\sigma$-fields and $(T_n)_{n\in\mn_0}$ be a sequence of random variables such that \begin{equation}\label{eq:cond_berry_esseen}
\sum_{n\geq 0}\underset{y\in\mr}{\sup} \left|\Prob [T_n\leq y|\mathcal{G}_n] - \Phi(y)\right|<\infty \quad\text{a.s.},
\end{equation}
where $\Phi(y)=\frac 1 {\sqrt{2\pi}} \int_{-\infty}^y \eee^{-x^2/2}{\rm d}x$, $y\in\mr$.
Then,
$$
\underset{n\to\infty}{\lim\sup} \frac {T_n}{\sqrt{2\log n}} \leq 1 \quad \text{a.s.}
$$
If, further, there is a $k\in\mn$ such that $T_n$ is $\mathcal{G}_{n+k}$-measurable for each $n\in\mn_0$, then
$$
\underset{n\to\infty}{\lim\sup} \frac {T_n}{\sqrt{2\log n}}=1 \quad \text{a.s.}
$$
\end{lem}
Let $r\in\mn_\infty$ be fixed. We are going to verify condition~\eqref{eq:cond_berry_esseen} for the random variables
$$
T_n := V_{n,r}/ \sqrt{\Var [V_{n,r}|\mathcal{F}_n]}.
$$
Conditionally given $\mathcal{F}_n$, $V_{n,r}$ is a weighted sum of i.i.d.\ random variables to which the Berry--Esseen inequality (see Lemma~\ref{lem_berry_esseen} below) applies:
\begin{eqnarray*}
\Delta_{n,r}
&:=&
\underset{y\in\mr}{\sup}\left|\Prob \left[\frac{V_{n,r}}{\sqrt{\Var [V_{n,r}|\mathcal{F}_n]}}\leq y\Bigg |\mathcal{F}_n\right] - \Phi(y)\right|
\\
&\leq&
C {\sum_{|u|=n} \me\left[\left|Z^{(u)}_{n,r}-\me[Z^{(u)}_{n,r}|\mathcal{F}_n]\right|^3|\mathcal{F}_n\right]\over (\Var[V_{n,r}|\mathcal{F}_n])^{3/2}}
\\
&\leq&
8C {\sum_{|u|=n} \me[|Z^{(u)}_{n,r}|^3|\mathcal{F}_n] \over (\Var[V_{n,r}|\mathcal{F}_n])^{3/2}},
\end{eqnarray*}
where $C>0$ is a finite absolute constant. Now we work towards
proving that
\begin{equation}\label{basic}
\sum_{n\geq 0}\Delta_{n,r}<\infty\quad \text{a.s.}
\end{equation}
which would verify condition~\eqref{eq:cond_berry_esseen}.
Equation~\eqref{important} reveals that \eqref{basic} would hold provided we could prove that $B<\infty$ a.s.,\ where
\begin{align}
B &:= \sum_{n\geq 0}\eee^{3an}\sum_{|u|=n}\me [|Z^{(u)}_{n,r}|^3|\mathcal{F}_n] \label{basic2}\\
&= \sum_{n\geq 0}\eee^{3an} \sum_{|u|=n}Y_u^3 \int_{[0,\infty)}x^3 \1_{\{\eee^{-an}Y_u^{-1}\geq x\}}{\rm d} F_r(x). \notag
\end{align}

To proceed, we need to define the random walk associated with the BRW. Consider the following probability measures on $\R$:
\begin{equation*}    \label{Sigman}
\overline{\Sigma}_n := \me \Bigg[\sum_{|u|=n} Y_u \, \delta_{S(u)}\Bigg],
\qquad n\in\mn.
\end{equation*}
The associated random walk $(S_n)_{n \in\mn_0}$ is a zero-delayed
random walk with increment distribution $\overline{\Sigma}_1$. It
is clear that, for any measurable $f: \R \to [0,\infty)$,
\begin{equation}    \label{eq:distribution_of_S_n}
\me f(S_n)  =
\me \Bigg[\sum_{|u|=n} Y_u
f(S(u))\Bigg],\quad n\in\mn.
\end{equation}
Passing to expectations in~\eqref{basic2} and
using~\eqref{eq:distribution_of_S_n} we obtain
\begin{eqnarray*}
\me B
&=&
\int_{[0,\infty)} x^3 \left(\sum_{n\geq 0}\eee^{an}\me\Big[\eee^{-2(S_n-an)}\1_{\{\eee^{S_n-an}>x\}}\Big]\right) {\rm d}F_r(x)
\\
&=&
\int_{[0,\infty)}x^3 \left(\int_x^\infty y^{-2}{\rm d}V(y) \right) {\rm d}F_r(x),
\end{eqnarray*}
where
\begin{equation}\label{vx}
V(x):=\sum_{n\geq 0}\eee^{an}\Prob\{S_n-an\leq \log x\},\quad x>0.
\end{equation}
By Lemma~\ref{Lem:asymptotics_of_V_a}, $V(x)<\infty$ for all $x>0$. Since the function $x\mapsto \int_x^\infty y^{-2}{\rm d}V(y)$ is nonincreasing we conclude, again by Lemma \ref{Lem:asymptotics_of_V_a}, that $\int_x^\infty y^{-2}{\rm d}V(y)\leq c/x$ for some constant $c>0$ and large enough $x$. Hence $\int_{(b,\infty)}x^3 \int_x^\infty y^{-2}{\rm d}V(y){\rm d}F_r(x)<\infty$ for any $b>0$ in view of
\begin{equation}\label{useful}
\Var W_r=\int_{[0,\infty)}x^2{\rm d}F_r(x)<\infty.
\end{equation}
We also have $\int_{[0,b]}x^3 \int_x^\infty y^{-2}{\rm d}V(y){\rm d}F_r(x)<\infty$ because $\underset{x\to 0+}{\lim} x^3\int_x^\infty y^{-2}{\rm d}V(y)=0$. To verify the latter relation, integrate by parts and apply L'H\^{o}spital's rule.
This proves that $B<\infty$ a.s.\ and thereupon~\eqref{basic}.

An appeal to Lemma \ref{asmus} with $T_n = V_{n,r}/ \sqrt{\Var [V_{n,r}|\mathcal{F}_n]}$ in combination with \eqref{important} leads to the conclusion: for fixed $r\in\mn$,
\begin{equation}\label{1}
\underset{n\to\infty}{\lim\sup}{\eee^{an} V_{n,r}\over\sqrt{2\log n }}=\sqrt{W_\infty(2)\Var W_r}\quad\text{a.s.}
\end{equation}
because $V_{n,r}$ is $\mathcal{F}_{n+r}$-measurable; whereas
\begin{equation}\label{2}
\underset{n\to\infty}{\lim\sup}{\eee^{an}V_{n,\infty}\over\sqrt{2\log n }}\leq \sqrt{W_\infty(2)\Var W_\infty}\quad\text{a.s.}
\end{equation}
Comparing formulae \eqref{repr1} and \eqref{repr2} we conclude that in order to show that \eqref{1} and \eqref{2} imply
\begin{equation}\label{3}
\underset{n\to\infty}{\lim\sup}{\eee^{an}(W_{n+r}-W_n)\over\sqrt{2\log n }}=\sqrt{W_\infty(2)\Var W_r}\quad\text{a.s.}
\end{equation}
and
\begin{equation}\label{4}
\underset{n\to\infty}{\lim\sup}{\eee^{an}(W_\infty-W_n) \over\sqrt{2\log n }}\leq \sqrt{W_\infty(2)\Var W_\infty}\quad\text{a.s.}
\end{equation}
it suffices to prove that, for $r\in\mn_\infty$,
\begin{equation}\label{eq:rel1}
\lin \eee^{an} \sum_{|u|=n}Y_u|W_r^{(u)}-1|\1_{\{\eee^{an} Y_u|W_r^{(u)}-1| > 1\}}=0 \quad\text{a.s.}
\end{equation}
and
\begin{equation}\label{eq:rel2}
\lin \eee^{an}\sum_{|u|=n}|\me [Z^{(u)}_{n,r}|\mathcal{F}_n]|=
0\quad\text{a.s.}
\end{equation}
Since $\me [W_r^{(u)}-1]=0$ and $Y_u$ is $\F_n$-measurable for $|u|=n$, we have
\begin{eqnarray*}
\big|\me[Z^{(u)}_{n,r}|\mathcal{F}_n]\big|
&=&
\left|\me \left[Y_u (W_r^{(u)}-1) \1_{\{\eee^{an}Y_u|W_r^{(u)}-1|\leq 1\}} \Big | \mathcal{F}_n\right]\right|
\\
&=&
\left|\me \left[Y_u (W_r^{(u)}-1) \1_{\{\eee^{an}Y_u|W_r^{(u)}-1|> 1\}} \Big|\mathcal{F}_n\right]\right|
\\
&\leq&
\me \left[ Y_u |W_r^{(u)}-1| \1_{\{\eee^{an}Y_u|W_r^{(u)}-1|> 1\}} \Big | \mathcal{F}_n\right].
\end{eqnarray*}
Hence both relations~\eqref{eq:rel1} and~\eqref{eq:rel2} follow if we can show that
$$
I:=
\me\Bigg[\sum_{n\geq 0}\eee^{an}\sum_{|u|=n}Y_u|W_r^{(u)}-1|\1_{\{\eee^{an}Y_u|W_r^{(u)}-1|>1\}}\Bigg]<\infty.
$$
Since $V$ is nondecreasing, an application of Lemma \ref{Lem:asymptotics_of_V_a} yields $V(x)\leq cx$ for some constant $c>0$ and large enough $x$. Using this we infer
$$
I=
\me \Bigg[\sum_{n\geq 0}\eee^{an}\sum_{|u|=n} Y_u\int_{[0,\infty)} x\1_{\{\eee^{-an}Y_u^{-1}\leq x\}}{\rm d}F_r(x)\Bigg]
=
\int_{[0,\infty)}x V(x){\rm d}F_r(x)
<\infty
$$
in view of \eqref{useful}. The proof of \eqref{3} and \eqref{4} is complete.

It remains to show that ``$\leq$'' can be replaced by ``$=$'' in~\eqref{4}. As has already been remarked at the beginning of the proof, once we have proved \eqref{4} we also have
\begin{equation}\label{5}
\underset{n\to\infty}{\lim\inf}{\eee^{an}(W_\infty-W_n) \over\sqrt{2\log n }}\geq -\sqrt{W_\infty(2)\Var W_\infty}\quad\text{a.s.}
\end{equation}
For any $r\in\mn$, the following equality holds $${\eee^{an}(W_\infty-W_n)\over \sqrt{\log n}}={\eee^{a(n+r)}(W_\infty-W_{n+r})\over \sqrt{\log (n+r)}}{\sqrt{\log(n+r)}\over \sqrt{\log n}}\eee^{-ar}+{\eee^{an}(W_{n+r}-W_n)\over \sqrt{\log n}}.$$ Using now \eqref{3} and \eqref{5} we infer
\begin{align*}
\lefteqn{\underset{n\to\infty}{\lim\sup}{\eee^{an}(W_\infty-W_n)\over \sqrt{\log n}}}\\
&\geq
\underset{n\to\infty}{\lim\inf} {\eee^{a(n+r)}(W_\infty-W_{n+r})\over \sqrt{\log (n+r)}}{\sqrt{\log(n+r)}\over \sqrt{\log n}}\eee^{-ar}+\underset{n\to\infty}{\lim\sup}{\eee^{an}(W_{n+r}-W_n)\over\sqrt{\log n}}\\
&\geq -\sqrt{2W_\infty(2)\Var W_\infty}\eee^{-ar}+\sqrt{2W_\infty(2)\Var W_r}.
\end{align*}
Letting $r\to\infty$ we arrive at
$$
\underset{n\to\infty}{\lim\sup}{\eee^{an}(W_\infty-W_n)\over \sqrt{\log n}}\geq \sqrt{2W_\infty(2)\Var W_\infty}.
$$
This completes the proof of Theorem \ref{main1}.

\section{Appendix}

The following result is concerned with the asymptotics of $V(x)$ defined in \eqref{vx}. This is a slightly extended specialization of Lemma 3.1 in \cite{Iksanov+Meiners:2010}.

\begin{lem}   \label{Lem:asymptotics_of_V_a}
Suppose that the function $r\to (m(r))^{1/r}$ decreases on $[1,2]$ and that~\eqref{mean} holds.
Then $V(x)<\infty$ for all $x>0$.  If, furthermore, the associated random walk $(S_n)_{n\in\mn_0}$ is non-arithmetic, then as $x \to \infty$,
\begin{equation}    \label{eq:asymptotic_of_V_a}
V(x)~\sim~ c_a x,
\end{equation}
where $c_a:=(\eee^{2a}(-m'(2))-a)^{-1}\in (0,\infty)$,
and
\begin{equation}
\label{eq:tail_asymptotics_of_V_a}
\int_{(x,\,\infty)} \!\! y^{-2} \, {\rm d}V(y) ~\sim~
c_a x^{-1}.
\end{equation}
If
$(S_n-an)_{n \in\mn_0}$ is arithmetic with span $\lambda_a > 0$, then,
analogously, as $n \to \infty$
\begin{eqnarray}    \label{eq:asymptotic_of_V_a_lattice}
V(\eee^{\lambda_a n})~ \sim~ d_a \eee^{\lambda_a n},
\end{eqnarray}
where $d_a:=\lambda_a((1-\eee^{-\lambda_a})(\eee^{2a}(-m^{\prime}(2)) -a))^{-1}\in (0,\infty)$,
and
\begin{equation}
\label{eq:tail_asymptotics_of_V_a_lattice}
\int_{[\eee^{\lambda_a n},\,\infty)} \!\! y^{-2}{\rm d}V(y) ~\sim~ d_a \eee^{-\lambda_a n}.
\end{equation}
\end{lem}
\begin{proof}
Formulae \eqref{eq:asymptotic_of_V_a} and \eqref{eq:asymptotic_of_V_a_lattice} are borrowed from Lemma 3.1 in \cite{Iksanov+Meiners:2010}.
Relation \eqref{eq:tail_asymptotics_of_V_a} follows from \eqref{eq:asymptotic_of_V_a} by integration by parts and subsequent application of Proposition 1.5.10 of \cite{Bingham+Goldie+Teugels:1989}. Relation \eqref{eq:tail_asymptotics_of_V_a_lattice} can be obtained with the help of elementary calculations in combination with $V(\eee^{\lambda_a n})-V(\eee^{\lambda_a(n-1)})~ \sim~ d_a(1-\eee^{-\lambda_a}) \eee^{\lambda_a n}$ which is a consequence of \eqref{eq:asymptotic_of_V_a_lattice}.
\end{proof}

Since we consider BRW in which particles are allowed to have infinite number of offspring with positive probability, we need a version of the Berry--Esseen inequality for sums with possibly infinite number of summands.
\begin{lem}\label{lem_berry_esseen}
Let $X_1, X_2,\ldots$ be independent (but not identically distributed) random variables with $\me X_i =0$, $\sigma^2_i := \Var X_i$ and $\rho_i := \me|X_i|^3$, $i\in\mn$. If $\sum_{i\geq 1} \sigma_i^2<\infty$, then, for an absolute constant $C$,
\begin{equation}\label{eq:berry_esseen}
\sup_{y\in\R} \left|\Prob\left[\frac{\sum_{i\geq 1} X_i}{\left(\sum_{i\geq 1} \sigma_i^2\right)^{1/2}}\leq y\right] - \frac 1 {\sqrt{2\pi}} \int_{-\infty}^y \eee^{-x^2/2}{\rm d}x \right| \leq C \frac{\sum_{i\geq 1} \rho_i} {\left(\sum_{i\geq 1} \sigma_i^2\right)^{3/2}}.
\end{equation}
\end{lem}
\begin{proof}
According to the classical Berry--Esseen inequality, \eqref{eq:berry_esseen} is valid if all infinite sums are replaced by finite sums over $i=1,\ldots,n$ with arbitrary $n\in\mn$. By letting in the classical inequality $n\to\infty$ and noting that $\eta_n := \sum_{i=1}^n X_i/ (\sum_{i=1}^{n} \sigma_i^2)^{1/2}$ converges to its infinite version $\eta_\infty$ a.s.\ (and hence in distribution), we obtain that~\eqref{eq:berry_esseen} holds for all $y$ which are continuity points of $\eta_\infty$. Since any $y\in\R$ can be approximated by continuity points from the right and since the distribution function is right-continuous, we obtain that~\eqref{eq:berry_esseen} holds for all $y\in\R$.
\end{proof}

\acks The authors thank the anonymous referee for several
valuable comments that helped improving the presentation.
Z.~Kabluchko is grateful to Pascal Maillard for a useful
discussion. A part of this work was done while A.~Iksanov was
visiting M\"{u}nster in July 2015. A.I. gratefully acknowledges
hospitality and the financial support by DFG SFB 878 ``Geometry,
Groups and Actions''.

\end{document}